\documentclass{amsart}

\usepackage{amssymb, mathrsfs, comment, enumerate, tikz, setspace, xparse, expl3}
\usepackage{hyperref}

\newtheorem{thm}{Theorem}[section]

\newtheorem{prop}[thm]{Proposition}
\newtheorem{lem}[thm]{Lemma}

\theoremstyle{remark}
\newtheorem{rem}[thm]{Remark}
\newtheorem{ex}[thm]{Example}

\theoremstyle{definition}
\newtheorem{dfn}[thm]{Definition}

\numberwithin{equation}{section}
\numberwithin{thm}{section}

\newcommand{\D}{\mathbb{D}}
\newcommand{\C}{\mathbb{C}}
\newcommand{\Ps}{\mathbb{CP}}
\newcommand{\sO}{\mathcal{O}}

\newcommand{\Hom}[1]{{\bf #1}}
\newcommand{\dbar}{\overline{\partial}}
\newcommand{\edge}{\sigma(j_1, \ldots, j_k)}
\newcommand{\e}{\epsilon}
\newcommand{\Leray}{\mathscr{L}(\Hom{z}, \boldsymbol{\tau})}
\newcommand{\pair}[1]{\langle \boldsymbol{\tau}, \Hom{#1} \rangle}
\newcommand{\weak}{W_{\Hom{z}}{(\Omega)}}

\begin{document}






\title{Projectively Invariant Hardy Spaces on Domains with Corners}


%

\author{Benjamin Krakoff}
\address{
}

\email{bkrakoff@umich.edu}


%

\thanks{}

%

\begin{abstract}
For smoothly bounded, strongly $\C$-convex domains, one can use the Fefferman form or its variants to define projectively invariant norms on sections of holomorphic line bundles, producing a Hardy space. In two variables, we construct a projectively invariant measure on the singular part of a piece-wise smooth domain, and show that positivity of this invariant coincides with a notion of strong $\C$-convexity that is compatible with Cauchy-Fantappi\'e-Leray kernels, and thus define projectively invariant Hardy spaces as in the smooth case.
\end{abstract}

\maketitle



\section{Introduction}
\par Given a domain $\Omega \subset \Ps^n$, $\Omega$ is \emph{$\C$-linearly convex} if the complement of $\Omega$ is a union of complex hyperplanes. If $\Omega$ has $C^2$-boundary $S$, this is equivalent to the complex tangent hyperplanes lying to one side of $\Omega$, in which case the second fundamental form is positive semi-definite on the complex tangent hyperplane and $\Omega$ is also $\C$-convex. If additionally these complex tangent planes intersect $S$ precisely at the point of tangency, we say $\Omega$ (or $S$) is \emph{strictly} $\C$-convex. Finally, if the second fundamental form is strictly positive definite on the maximal complex subspace, then we say $\Omega$ (or $S$) is \emph{strongly} $\C$-convex.
\par For strongly $\C$-convex $\Omega$, it is known how to construct projectively invariant Hardy spaces that satisfy many of the properties familiar from the one variable theory. Using the Fefferman form, or Barrett's preferred measure, one can construct projectively invariant norms on sections of line bundles over $S$, and thus define invariant $L^2$ and Hardy spaces. For holomorphic sections, one can use Cauchy-Fantappi\'e-Leray integrals to reconstruct values on $\Omega$ from the values on the boundary $S$. In this context, these Cauchy-Fantappi\'e-Leray kernels are simply the Leray transform, and one can make sense of the corresponding singular integrals to turn the Hardy spaces into a reproducing kernel Hilbert space.
\par Our goal in this paper is to define projectively invariant Hardy spaces on a pseudoconvex domain $\Omega$ with piece-wise smooth boundary $S$ in 2 variables. To achieve this, we first define an appropriate notion of strict $\C$-convexity in $n$ variables defined in terms of a generalization of complex tangents, and use Cauchy-Fantappi\'e-Leray integrals to produce a holomorphic reproducing kernel. This reproducing kernel is the standard Leray kernel on the smooth part of $S$ and gives positive mass to the singular part of $S$. Next restricting our attention to 2 variables, we define an appropriate bundle-valued invariant which is non-negative when the domain is $\C$-convex, and when this invariant is positive we show $\Omega$ satisfies some local notion of strong $\C$-convexity. Lastly, we use this invariant to produce a norm and a corresponding Hardy space $H^2(\Omega)$ on $S$ and show that, via the reproducing kernel, one can recover the values of a holomorphic section $f$ via its corresponding element in $H^2(\Omega)$.

\section{Projective Transformations and Line Bundles}
\subsection{The Bundles $O_E(j,k)$ and $\sO_E(j,0)$} \hfill
\par We closely follow the expositions in \cite{barrett_2015} section 2 and \cite{APS} section 3.2. Let $\Ps^n$ have projective coordinates $\Hom{z} = [z_0:\ldots:z_n]$. $z = (\frac{z_1}{z_0}, \ldots, \frac{z_n}{z_0})$ be standard affine coordinates, setting $\hat{z}_j = \frac{z_j}{z_0}, \ j \neq 0$.
\par Let $T: \Ps^n \rightarrow \Ps^n$ be a projective transformation. Choose an invertible matrix $(M_{i,j})_{i,j=0}^n$ with $\det(M)=1$ which descends to $T$. Any two choices of $M$ differ by a $n+1$ root of unity, and for the formulas we care about the ambiguity will wash out. In the standard affine coordinates,
\[T(\hat{z}_1, \ldots, \hat{z}_n) = \frac{(M_{1,0}+\sum_{j=1}^n M_{1,j}\hat{z}_j, \ldots, M_{n,0} + \sum_{j=1}^n M_{n,j}\hat{z}_j)}{M_{0,0} + \sum_{j=1}^n M_{0,j}\hat{z}_j}.\]
 Given a set $E \subset \Ps^n$ and integers $j,k$, let $O_E(j,k)$ denote the space of continuous functions on the cone over $E$ satisfying
 \[F(\lambda z_0, \ldots, \lambda z_n) = \lambda^{j}\overline{\lambda^k} F(z_0, \ldots z_n).\]
In the standard affine coordinates, we can identify $F$ with the continuous function $f$ on $E$, via $f(\hat{z}_1, \ldots, \hat{z}_n) = F(1, \hat{z}_1, \ldots, \hat{z}_n)$ and obtain the following transformation law
 \[ (T^*f)(z) = (M_{0,0} + \sum_{i=1}^n M_{0,i}\hat{z}_i)^j \overline{(M_{0,0} + \sum_{i=1}^n M_{0,i}\hat{z}_i)^k} f(T(z)). \]\[\]
 \par If $E$ is open, $\sO_E(j,0)$ will denote the space of holomorphic sections on $E$, and if $E$ is closed $\sO_E(j,0)$ will denote the space of sections holomorphic in some neighborhood of $E$. With this convention, expressions like $z_j$ can be thought of as sections of $\sO_{\Ps^n}(1, 0)$. Given $F \in O_{E}(j,k)$, we have well-defined expressions $\overline{F} \in O_E(k,j), |F|^2 \in O_{E}(j+k,j+k)$, etc. More generally, if $G \in O_{E}(j', k')$ then $FG \in O_E(j,k) \otimes O_E(j',k') = O_E(j+j', k+k')$. When $j = k$, one can make sense of the section $F$ being positive.  Observe that a section of $O_E(0,0)$ is nothing but a continuous function on $E$.
 \par We will need the notion of a projectively invariant, bundle-valued form. For example, the differential form $dz := d\hat{z}_1 \wedge \ldots \wedge d\hat{z}_n$ satisfies the transformation law
 \[ T^*dz = (M_{0,0} + \sum_{j=1}^n M_{0,j}\hat{z}_j)^{-n-1} dz \]
 Thus, the expression $z_0^{n+1} dz$ is a projectively invariant $\sO_{\Ps^n}(n+1,0)$-valued $n$-form. More generally, an $L$-valued $k$-form on $M$ is a section of $L \otimes \bigwedge^k T^*(M)$. Pulling back a section means writing as a product of a section $s$ of $L$ and $k$-form $\omega$ and pulling back each.
 \subsection{Duality and The Universal Cauchy-Fantappi\'e-Leray Form} \hfill
 \par We give a brief overview of duality theory for $\C$-convex domains, since it is required for constructing the desired reproducing kernels in section \ref{sec:LerayKernel}. See \cite{APS} Chapter 3 for more details of the theory.
 \par The dual projective space $(\Ps^n)^*$ denotes the set of complex hyperplanes in $\Ps^n$. We may give $(\Ps^n)^*$ homogeneous coordinates $\Hom{w} = [w_0: \ldots : w_n]$ by identifying $\Hom{w}$ with the hyperplane $\ell_{\Hom{w}} := \{\Hom{z} | \Hom{w} \cdot \Hom{z} = 0\}$. We may also use the notation $\langle \Hom{w}, \Hom{z} \rangle$ to indicate dot product, depending on which is easiest to read in context. The \emph{incidence locus} $I \subset \Ps^n \times (\Ps^n)^*$ is the set of points incident to hyperplanes.
 \[ I := \{(\Hom{z}, \Hom{w}) | \Hom{z} \cdot \Hom{w} = 0\} \]
 Given a projective transformation $T$, there is a unique dual transformation $T^{-t}$ preserving $I$. If $T$ is given by the matrix $M$, then $T^{-t}$ is given by the matrix $M^{-t} := {}^t M^{-1}$, explaining the notation. For $\Omega$ $\C$-convex with (not necessarily smooth) boundary $S$, denote by $S^*$ the set of hyperplanes $\Hom{w}$ incident to $S$ not intersecting $S^*$. When $S$ is $C^2$ and strongly $\C$-convex, this is the dual hypersurface studied in \cite{barrett_2015}. Let $I_S$ be the set of pointed hyperplanes incident to $S$ but avoiding $\Omega$,
 \[ I_S := I \cap (S \times S^*) \]
 \par On $I$, the \emph{Universal Cauchy-Fantappi\'e-Leray form} is the unique (up to constants) projectively invariant $\sO_{\Ps^n}(n,0) \otimes \sO_{(\Ps^n)^*}(n, 0)$-valued $2n-1$ form. On the coordinate patch $U_{j,k} := \{z_j \neq 0, w_j \neq 0\}$, it is given by the expression
 \[ \omega_{CFL} = \frac{z_j^n w_k^n}{(2\pi i)^n} \left( \sum_{l=0}^n \frac{z_l}{z_j} d \left( \frac{w_l}{w_k}\right) \right) \wedge \left(\sum_{l=0}^{n} d \left( \frac{z_l}{z_j} \right) \wedge d\left( \frac{w_l}{w_k}\right) \right)^{n-1} \]
Keeping in mind the relation $\Hom{z}\cdot \Hom{w} = 0$, it's not hard to see that the Universal CFL form has the symmetry property $\omega_{CFL}(\Hom{z}, \Hom{w}) = (-1)^n \omega_{CFL}(\Hom{w}, \Hom{z})$. This form, along with the following theorem, gives a recipe for constructing holomorphic reproducing kernels on $\Omega$. 
 \begin{thm}(\cite{APS}, Thm. 3.1.7, \cite{AY} Lemma 3.3 and Corollary 3.6) \hfill \\
 \label{thm:APS}
 Let $\Omega$ be $\C$-convex such that $S$ and $S^*$ are smooth cycles. For $\boldsymbol{\tau} \in \Ps^n$, set $g_{\boldsymbol{\tau}}(\Hom{w}) := \frac{1}{\langle \boldsymbol{\tau}, \Hom{w}\rangle^n}$. Then for $f \in \sO_{\overline{\Omega}}(-n, 0), \ \boldsymbol{\tau} \in \Omega$
 \[  f(\boldsymbol{\tau}) = \int_{I_S} f \ g_{\boldsymbol{\tau}} \ \omega_{CFL} \]
 \end{thm}
 \par When $S$ is $C^2$, the only complex hyperplane incident to $\Hom{z} \in S$ avoiding $\Omega$ is the maximal complex subspace. Thus $\Hom{w}$ is a function of $\Hom{z}$, and given a defining function $\rho$ we have
 \[ \Hom{w} = [-\langle \partial \rho,z \rangle: \partial \rho] \]
 Plugging this expression into  the integral of Thm. \ref{thm:APS} and identifying sections of line bundles with functions produces the \emph{Leray kernel}
 \[ L\{f\}(\tau) := \frac{1}{(2\pi i)^n} \int_S f(z) \frac{\partial \rho \wedge (\dbar \partial \rho)^{n-1}}{\langle \partial \rho, z - \tau \rangle^n} \] 
 and the theorem implies that the Leray kernel reproduces values of holomorphic functions on $\Omega$ with controlled boundary values.
 \subsection{The Fefferman Form, Barrett's Preferred Measure}\hfill
\par Let $\Omega \subset \Ps^n$ be strongly $\C$-convex. Given a defining function $\rho$ for $S$ so that $\rho < 0$ on the pseudoconvex side, the \emph{Fefferman form} (\cite{Fefferman_1979}, p. 259) is a $2n-1$-form on $S$ by the identity 
$$ \mu_{S, Fef} \wedge d\rho = 2^{\frac{2n}{n+1}} (-1)^n \begin{vmatrix} 0 & \rho_{\overline{z_j}} \\
\rho_{z_k} & \rho_{z_k, \overline{z_j}}
\end{vmatrix}^{\frac{1}{n+1}} \omega_{\C^n}$$
where $\omega_{\C^n}$ is the Euclidean area form on $\C^n$. Positivity of the Fefferman form follows from the positivity of the Levi form. Given a holomorphic transformation $F: S \rightarrow F(S)$ we have
\[ F^*(\mu_{F(S), Fef}) = |\det(F')|^{\frac{2n}{n+1}} \mu_{S, Fef} \]
In the specific case of $F$ a projective transformation, the above identity implies that the expression $|z_0|^{2n}\mu_{S, Fef}$ is a projectively invariant $O_S(n,n)$-valued $2n-1$ form. For convenience, we will write just $\mu_{S, Fef}$ to mean the bundle-valued form.
\par The Fefferman form defines an invariant norm on sections of $O_S(-n,0)$ via
\[ ||f||^2 = \int_S |f|^2 \mu_{S, Fef} \]
Note that the integrand is an $O_{S}(-n,0) \otimes O_S(0,-n) \otimes O_S(n,n)$-valued $2n-1$ form, in other words an ordinary $2n-1$ form.
\par For studying $S$ along with its dual hypersurface $S^*$, one should use Barrett's preferred measure $\mu_{S, Bar}$, which is the Fefferman form multiplied by a projectively invariant scalar. This measure factorizes the Universal CFL form,
\[ |\omega| = \mu_{S, Bar}^{\frac{1}{2}} \ \mu_{S^*, Bar}^{\frac{1}{2}} \]
and so connects more optimally with the duality pairing between $\Omega$ and its dual. See section 8 of \cite{barrett_2015}. 
\par Using a choice of projectively invariant norm $\mu$, one can complete $O_{S}(-n,0)$ to the $L^2$-space $L^2(S, \mu)$, and define the corresponding Hardy space as the $L^2$-closure of restrictions of sections holomorphic in a neighborhood of $\Omega$
\[ H^2(\Omega, \mu) := \overline{\{f|_S| f\in \sO_{\overline{\Omega}}(-n,0)\}} \]
For $\boldsymbol{\tau} \in \Omega$, the Leray kernel directly recovers the values of functions in $H^2(\Omega, \mu)$. For $\boldsymbol{\tau} \in S$, one must either take a non-tangential limit, or take the principal value of a singular integral. For proofs in one variable, see \cite{Duren_1970} Chapter 1 and \cite{Bell_1992} Chapter 5. In fact, more is true. These singular integrals give bounded maps $L^2(S, \mu) \rightarrow H^2(\Omega, \mu)$. In one variable, see the seminal work of \cite{Calderon_1977} and \cite{Coifman_McIntosh_Meyer_1982}. For proofs in several variables, as well as an excellent survey, see \cite{Lanzani_Stein_2013}.

\section{Piece-wise Smooth Domains}
\par Let $\Omega \subset \Ps^n$ be an open domain. The boundary $S$ is \emph{piece-wise smooth} if 
\begin{enumerate}[(a)]
    \item There are finitely many disjoint, connected $C^2$ hypersurfaces $H_j$ so that $S \subset \cup H_j$
    \item For $k \leq n$, every set of $k$ boundary hypersurfaces $H_j$ meets complex transversely.
    \item For $k > n$, every set of $k$ boundary hypersurfaces meets real transversely.
    \item For each $\Hom{z} \in S$ and sufficiently small neighborhood $U \ni \Hom{z}$, $\Omega \cap U$ is connected.
\end{enumerate}
\par Item $c)$ implies that no more that $2n$ hypersurfaces can meet at one point. Given $k$ hypersurfaces $H_{j_1}, \ldots, H_{j_k}$, the \emph{edge} $\edge$ is where exactly these $k$ hypersurfaces intersect and no others. It follows that $\edge$ is a manifold of real dimension $2n-k$, which is totally real for $k \geq n$.
\begin{lem}
\label{lem:TCconvex}
Let $\Omega \subset \Ps^n$ be a piece-wise smooth, pseudoconvex domain, $n \geq 2$. For $\Hom{z} \in \edge$, let $\rho_{j_l}$ be local defining functions for $H_{j_l}$ such that $\rho_{j_l} < 0$ on part of $\Omega$. Then for a sufficiently small neighborhood $U \ni \Hom{z}$, $\Omega \cap U = \cap_{l=1}^k \{\rho_{j_l} < 0\} \cap U$.
\end{lem}
One way to interpret the above result is that the \emph{tangent cone} of $\overline{\Omega}$ at $\Hom{z}$ is convex. The main idea of the proof is to use the Kontinuit\"atssatz, see Figure \ref{fig:boundarywedge}.
\begin{figure}
	\begin{center}
		\begin{tikzpicture}
			\draw (-2, 0) -- (0,0);
			\draw (0,-2) -- (0,0);
			\draw[dashed] (0,0) -- (2,0);
			\draw[dashed] (0,0) -- (0,2);
			\draw (2,0) circle (0pt) node[anchor = west]{$\rho_1$};
			\draw (0,-2) circle (0pt) node[anchor = north]{$\rho_2$};
			\draw (-1,-1) circle (0pt) node{$\Omega$};
			
			\draw[dashed] (4, 0) -- (6,0);
			\draw[dashed] (6,-2) -- (6,0);
			\draw (6,0) -- (8,0);
			\draw (6,0) -- (6,2);
			\draw (8,0) circle (0pt) node[anchor = west]{$\rho_1$};
			\draw (6,-2) circle (0pt) node[anchor = north]{$\rho_2$};
			\draw (5,-1) circle (0pt) node{$\Omega$};
			
			\draw[red] (5,0) -- (6,-1);
			\draw[red] (5,0) -- (5.9, .9);
			\draw[red] (6,-1) -- (6.9,-.1);
			\draw[red, dashed] (5.9,.9) -- (6.9,-.1);
		\end{tikzpicture}
		\caption{$\Omega$ piece-wise smooth and pseudoconvex on the left, $\Omega$ piece-wise smooth and not pseudoconvex the right. A boundary of a family of holomorphic disks is in solid red, with the dashed line representing the final disk leaving the domain $\Omega$.}
		\label{fig:boundarywedge}
	\end{center}
\end{figure}
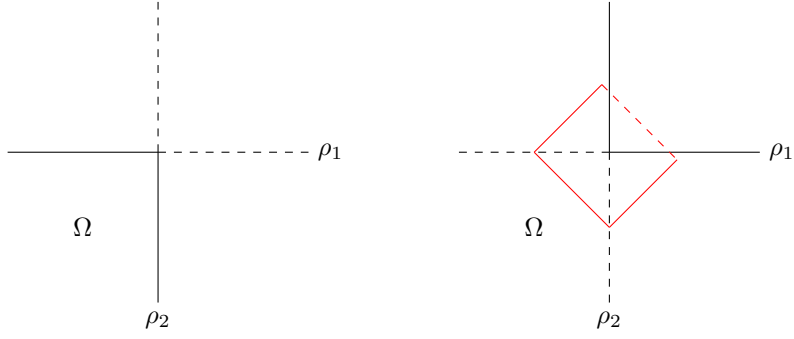

\begin{proof}
We first prove the lemma when $\Hom{z}$ is in an edge of dimension $2n-2$, that is there are precisely two boundary hypersurfaces $H_1$ and $H_2$ meeting complex transversely at $\Hom{z}$.
\par In this sub-case, we will also first prove the statement when $n=2$. The idea is to produce a family of holomorphic disks $F_t: \overline{\D} \rightarrow \C^2, 0 \leq t \leq 1$ such that $F_0(\D) \subset \subset \Omega$ and $F_t(\partial \D) \subset \subset \Omega$ for all $t$ while $F_1(\D) \cap int(\Omega^c) \neq \emptyset$, violating the so-called Kontinuit\"atssatz (see Theorem 3.3.5 of \cite{krantz_2008}).
\par By changing coordinates, assume $z = 0$ and  $T_{0}(H_j) = \{Im(\hat{z}_j) = 0\}$, the neighborhood $U$ is a ball of radius 1, $\max_{z \in S} dist(z, T_0(H_1) \cup T_0(H_2))$ is arbitrarily small. We construct a family $F_t$ as above for when $\Omega$ is the complement of $\{Im(\hat{z}_1) \geq 0, Im(\hat{z}_2) \geq 0\}$, and it follows from our choice of coordinates that $F_t$ will have the desired properties for $\Omega$ as well.
\par Consider four points $p_1, \ldots, p_4$ on $\partial \D$ in clockwise order. Take two line segments, $C_1$ from $p_1$ to $p_4$ and $C_2$ from $p_2$ to $p_3$. Let $C_1(t), C_2(t), 0 \leq t \leq 1$ be homotopies of $C_1, C_2$ respectively, relative to each endpoint such that $C_j(0) = C_j$, $C_j(t)$ is a segment of a circle for each $t > 0$ and the $C_j(1)$ cross transversely in the interior of the $\D$ bounding a set $E$ compactly contained in $\D$. Let $f_1(z, t)$ be the M\"obius transformation so that $Im(f_1(C_1(t),t)) = 0$ and $Im(f_1(z,t)) < 0$ below $C_1(t)$. Let $f_2(z, t)$ be the M\"obius transformation so that $Im(f_2(C_2(t), t)) = 0$ and $Im(f_2(z, t)) < 0$ above $C_2(t)$. It follows that $F_t(z) := (f_1(z, t), f_2(z, t))$ has the desired properties. See Figure \ref{fig:analyticdisks} for construction of the family of analytic disks.
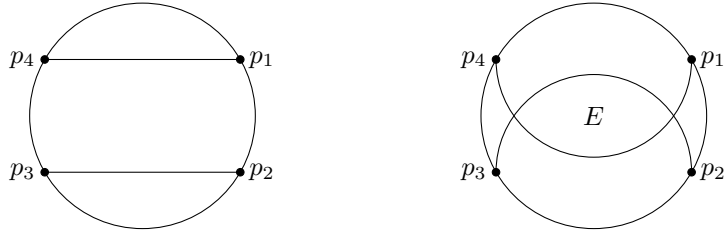
\begin{figure}
	\begin{center}
	\begin{tikzpicture}
		\draw (-3, 0) circle (1.5);
		\filldraw (-1.7, .75) circle (.05) node[anchor = west]{$p_1$};
		\filldraw (-1.7, -.75) circle (.05) node[anchor = west]{$p_2$};
		\filldraw (-4.3, -.75) circle (.05) node[anchor = east]{$p_3$};
		\filldraw (-4.3, .75) circle (.05) node[anchor = east]{$p_4$};
		\draw (-1.7, .75) -- (-4.3, .75);
		\draw (-1.7, -.75) -- (-4.3, -.75);
		
		\draw (3, 0) circle (1.5);
		\node at (3,0) {$E$};
		\filldraw (4.3, .75) circle (.05) node[anchor = west]{$p_1$};
		\filldraw (4.3, -.75) circle (.05) node[anchor = west]{$p_2$};
		\filldraw (1.7, -.75) circle (.05) node[anchor = east]{$p_3$};
		\filldraw (1.7, .75) circle (.05) node[anchor = east]{$p_4$};
		\draw (1.7, .75) arc (180:360:1.3);
		\draw (4.3, -.75) arc (0:180:1.3);
	\end{tikzpicture}
	\caption{Construction of the family of holomorphic disks, $F_0$ on the left, $F_1$ on the right. The set $E$ has $Im(f_j(E),1) > 0, \ j = 1,2$ and thus extends outside $\Omega$. Compare with Figure \ref{fig:boundarywedge}}
	\label{fig:analyticdisks}
	\end{center}
\end{figure}
	\par The sub-case when $n\neq2$ follows by considering the intersection of $\Omega$ with a complex plane of dimension $2$ which is transverse to $H_1$ and $H_2$. 
	\par We are now ready to prove the statement for $\Hom{z}$ in an arbitrary edge. The idea is to show that if the claim is violated for $\Hom{z}$ in an arbitrary edge, then it is also violated for a neighboring edge of dimension $2n-2$.
	\par Assume $z = 0$, and each $H_l$ has defining function $\rho_l$. Let $U$ be sufficiently small so that $\Omega \cap U$ is connected. Since all the $H_l$ are real-transverse, $\cup_{l=1}^k H_l$ divides $U$ into $2^k$ components $\{z | sgn(\rho_l(z)) = \e_l\}, \e \in \{-1, 1\}^k$. Let $C_k$ denote the hypercube graph on $\{-1, 1\}^k$, and let $G(\Omega \cap U)$ be the induced graph on the vertices $\e$ such that the component of $U \backslash \cup_{l=1}^k \{\rho_l = 0\}$ corresponding to $\e$ is contained in $\Omega$. Let $C$ denote the smallest sub-hypercube containing $G(\Omega \cap U)$. Our goal is to show that $C$ consists of precisely one vertex, so assume this is not the case.
	\par Since $\Omega \cap U$ is connected, $G(\Omega \cap U)$ is connected. Since $G(\Omega \cap U)$ cannot be all of $C$, (otherwise there are extraneous $H_l$ that we could discard) there must be two adjacent vertices in $C$ such that one vertex is in $G(\Omega \cap U)$ and one is not. By applying an automorphism of the cube we may assume that the all $1$'s vector $\Hom{1} \in G(\Omega \cap U)$ and $(-1, 1, \ldots, 1) \not \in G(\Omega \cap U)$. Since $G(\Omega \cap U)$ is contained in no smaller hypercube, $G' := G(\Omega \cap U) \cap \{ \e | \e_1 = -1\} \neq \emptyset$. Let $\gamma$ be a shortest path from $\Hom{1}$ to $G'$. $\gamma$ must contain at least 3 vertices, and consider the end of the path. Up to permuting coordinates, the last three vertices must be 
	 \begin{align*}
	 	(1, -x_2, x_3,\ldots, x_k) \rightarrow (1, x_2, x_3,  \ldots, x_k) \rightarrow (-1, x_2, x_3, x_4, \ldots, x_k).
	 \end{align*}
 The vertex $v^4:=(-1, -x_2, x_3, \ldots, x_k)$ cannot be in $G(\Omega \cap U)$ since then there would be a shorter path to $G'$. Let $v^1, v^2, v^3$ be the three last vertices of $\gamma$. By definition then $\Omega$ contains the components $\{sgn(z) = v^j\}, j=1, \ldots, 3$ and not $\{sgn(z) = v^4\}$. It follows that $H_1 \cap H_2$ is an edge contained in $S$ also violating the conclusion of the lemma, placing us back in the previous sub-case.
\end{proof}
\par Now that we have a rough description of the geometry our domains will satisfy, we need the concept of a weak tangent in order to eventually define our required notion of convexity. Denote by $TC_{\Hom{z}}(E)$ the \emph{tangent cone} of a set $E$ at $\Hom{z}$.
\begin{dfn}{Strong and Weak Tangents} \hfill
\label{def:weaktangents}
Let $\Omega \subset \Ps^n$ be a pseudoconvex, piece-wise smooth domain. $\Hom{w}$ is a \emph{strong tangent} to $\Hom{z} \in \edge$ if $\ell_{\Hom{w}}$ is the maximal complex subspace to one of the $H_{j_l}$ at $\Hom{z}$. $\Hom{w}$ is an \emph{interior weak tangent} if 
\[ TC_{\Hom{z}}(\ell_{\Hom{w}} \cap \Omega) \subset \edge \]
$\Hom{w}$ is a \emph{weak tangent} if it is in the closure of the interior weak tangents and is not a strong tangent. Denote by $W_{\Hom{z}}(\Omega)$ the strong and weak tangents of $\Omega$ at $\Hom{z}$.
\end{dfn}
In particular, if a hyperplane $\ell_{\Hom{w}}$ through $\Hom{z}$ avoids $\Omega$ then $\Hom{w}$ is a weak tangent.
\begin{prop}{Dual Cycles} \hfill 
\label{prop:weaktangents}
\par Let $\Omega$ be a pseudoconvex domain with piece-wise smooth boundary, $\Hom{z} \in \edge$ and $\Delta^k$ the standard $k-1$-simplex with $k$ vertices. We can parameterize $W_{\Hom{z}}(\Omega)$ via
\begin{align*}
	\Delta^k & \rightarrow W_{\Hom{z}}(\Omega) \\
	t & \mapsto [-\sum_{l} t_l \partial \rho_{j_l} \cdot z : \sum_{l} t_l \partial \rho_{j_l}]
\end{align*}
In particular, the vertices of $\Delta^k$ map to the strong tangents at $\Hom{z}$, and the union of the strong and weak tangents give a smooth cycle in $(\Ps^n)^*$.
\end{prop}

\begin{proof}
First we prove the statement for a convex cone. 
\par Assume that $z = 0$ (so $\Hom{z} = [1:0:\ldots:0]$), and suppose $\Omega = \{Im(\alpha_l \cdot z) < 0 , l = 1, \ldots, k\}$. Then the weak tangents through $0$ coincide with the positive cone over the $\alpha_l$, 
	\[  C := \{[0:w] | w =  \sum t_l \alpha_l, t_l \geq 0\} \]
\par Suppose we take a line of the form $\sum t_l \alpha_l \cdot z = 0$, $t_l \geq 0$, then we have that $\sum t_l Im(\alpha_l \cdot z) = 0$, so for at least one $l$, $Im(\alpha_l \cdot z) \leq 0$ and the line avoids $\Omega$.
\par Now suppose that we take a line $\alpha$ which is not in $C$. Since $C$ is convex, by Hahn-Banach there is a real hyperplane $\ell_{z} := \{ w | Im( z \cdot w) = 0 \}$ so that $Im( z \cdot C) < 0$ and $Im( z \cdot \alpha) > 0$. The first inequality implies that $z \in \Omega$, and the second implies that $\alpha$ cannot be written as a positive linear combination of the $\alpha_l$.
\par Now we prove the statement for general piece-wise smooth domains. Let $\Hom{z} \in \edge$. For each $l=1, \ldots, k$, write 
\[ \rho_{j_l} = -Im(\alpha_l \cdot z) + h.o.t. \]
For each $l$, let $P_{\alpha_l}$ denote the orthogonal projection onto the real hyperplane $T_{0}(H_{j_l})$. We claim that there are constants $C > 0, 0 < c << 1$ such that if $||z|| < c$ and $Im(\alpha_l \cdot z) > C ||P_{\alpha_l} z||^2$ for one $l$ then $z \not \in \Omega$.
\par We prove this for a fixed $l$, then the constants $C$ and $c$ can be chosen to be the max and min, respectively, of the all the chosen constants. Fixing $l$, choose an orthogonal change of coordinates so that $T_{0}(H_{j_l}) = \C^{n-1} \times \mathbb{R}$. Thinking of $H_{j_l}$ as a graph over $\hat{z}_1, \ldots, \hat{z}_{n-1}, Re(\hat{z}_n)$, we have up to second order
\[ Im(\hat{z}_n) = \sum_{i,j=1}^{n-1} L_{i,j}\hat{z}_i\overline{\hat{z}_j} + Re(Q_{i,j} \hat{z}_i\hat{z}_j) + \sum_{i=1}^{n-1} Im(R_i\hat{z}_i) \cdot Re(\hat{z}_n) + \tilde{R} \cdot Re(\hat{z}_n)^2 \]
The right hand side is a quadratic form in $z_1, \ldots, z_{n-1}, Re(z_{n})$, so taking $C$ larger than the largest eigenvalue of this quadratic form suffices.
\par Now suppose that we take $\Hom{w}$ to be in the interior of the cone over the $\alpha_k$, i.e. of the form $\Hom{w} = [0: w], \ w = \sum_l t_l \alpha_l$, $t_l > 0$. We want to show that $TC_{0}(\Omega \cap \ell_{\Hom{w}}) \subset T_{0}(\edge) = \{Im( \alpha_l \cdot z) = 0, l = 1, \ldots, k\}$. Suppose there were a sequence $z_s \rightarrow 0$ with $z_s \cdot w = 0$, $z_s \in \Omega$ and $\frac{z_s}{|z_s|} \rightarrow p \not \in \{Im(\alpha_l \cdot z) = 0, l = 1, \ldots, k\}$. It follows from $w \cdot p = 0$ that there is at least one $l$ with $Im(\alpha_l \cdot p) > 0$. We then have for $0 < t < 1$ and a constant $K$ that $Im(\alpha_l \cdot t \cdot p) > t \cdot K  \cdot ||p||$. It then follows for $s >> 1$ that
\[ Im(\alpha_l \cdot z_s) >  \tilde{K}  \cdot ||z_s|| > C ||P_{\alpha_l}(z_s)||^2 \] 
contradicting the fact that $z_s \in \Omega$. 
\par Now take $\Hom{w} = [0:w], \ w$ not in the closure of the positive cone over the $\alpha_j$. We know from the case of $\Omega$ a cone that $\ell_{\Hom{w}}$ will intersect the interior of the cone $\{Im(\alpha_l \cdot z) \leq 0 \}$, and from here it's easy to see that $TC_{0}(\Omega \cap \ell_{\Hom{w}})$ is not contained in $T_{0}(\edge)$. (For example, take a complex line $L \subset \ell_{\Hom{w}}$ intersecting the interior of the cone. The intersection of $L$ with the cone is a cone in $\C$ agreeing up to first order with $L \cap \Omega$).
\par Completing the proof just requires shifting the origin back to the original point $z$. The induced transformation in $(\Ps^n)^*$ is $[0:w] \mapsto [-w \cdot z: w]$.
\end{proof}

\section{A Leray Kernel}
\label{sec:LerayKernel}
\par We want to use the notion of weak tangents along with Thm. \ref{thm:APS} to formulate a holomorphic reproducing kernel (i.e. a Leray kernel) for piece-wise smooth, pseudoconvex domains. To this end, we say a piece-wise smooth $\C$-convex domain is \emph{strictly $\C$-convex} if every weak tangent intersects $\overline{\Omega}$ at exactly one point. Let $\pi: I_S \rightarrow S$ be projection onto the $\Hom{z}$-coordinates, keeping in mind that for piece-wise smooth strictly $\C$-convex domains, $\pi^{-1}(\Hom{z}) = \{(\Hom{z}, \Hom{w}) | \Hom{w} \in W_{\Hom{z}}(\Omega)\}$. Over each edge $\sigma := \edge$, write $\omega_{CFL} = \pi^*\eta_{\sigma} \wedge \tau_\sigma$, where $\eta_{\sigma}$ is a $2n-k$-form on $\sigma$, and $\tau$ is a $k-1$-form on $\pi^{-1}(\sigma)$. A generalization of the Leray kernel is the push-forward of the kernel of Thm \ref{thm:APS} by $\pi$, in other words
\[ \Leray = \left( \int_{W_{\Hom{z}}(\Omega)} \frac{1}{\langle \boldsymbol{\tau}, \Hom{w}\rangle^n} \tau_\sigma \right) \eta_\sigma \]
That this operation is well-defined is a small adaptation of the argument on page 61 of \cite{Bott_Tu_1982} for bundle-valued forms, involving applying the decomposition $\omega_{CFL} = \pi^* \eta_\sigma \wedge \tau_\sigma$ in two different sets of affinizations and checking that they differ by the correct factor.
\begin{thm}
\label{thm:LerayKernel}
Let $\Omega$ be a strictly $\C$-convex domain. Then on each edge $\sigma$ of dimension $2n-k$, $\Leray$ defines a projectively invariant $O_{\sigma}(n,0)$-valued $2n-k$-form with the reproducing property
\[ f(\boldsymbol{\tau}) = \int_S f(\Hom{z}) \Leray, \ f \in \sO_{\overline{\Omega}}(-n,0), \ \boldsymbol{\tau} \in \Omega \]
\end{thm}
Note that this integral gives ``weight" to edges of dimension $< 2n-1$. We show that the edges of dimension $<n$ have no weight (that is $\Leray$ vanishes) and compute an exact expression for edges of dimension $n$. Set $dz := d\hat{z}_1 \wedge \ldots d\hat{z}_n$, and set $dz_{[k]}$ to be the same with the term $d\hat{z}_k$ omitted.
\begin{prop}
\label{prop:LerayCalc}
Let $\Hom{z} \in \sigma$. If $\dim(\sigma) < n$, then $\Leray = 0$. If $\dim(\sigma) = n$, let $\Hom{w}^j, j = 1, \ldots, n$ denote the strong tangents at $\Hom{z}$, properly enumerated, and let $\det_k(\Hom{w}^j)$ denote determinant of the strong tangents with the $k^{th}$ coordinate omitted. Then
\[ \Leray = \frac{1}{(2\pi i)^n}\frac{\det_0(\Hom{w}^j)}{\prod_j \langle \boldsymbol{\tau}, \Hom{w}^j \rangle} z_0^n dz \]
\end{prop}

\begin{proof}{Proof of Thm. \ref{thm:LerayKernel}} \\
Strict convexity implies that $\pi$ is a submersion with fiber $W_{\Hom{z}}(\Omega)$, so we may use Fubini to write
\[ \int_{I_S} f(\Hom{z}) \frac{1}{\pair{w}^n} \omega_{CFL}  = \int_{S} f(\Hom{z}) \left( \int_{\weak} \frac{1}{\pair{w}^n}  \tau_\sigma \right) \eta_\sigma \]
Projective invariance of $\Leray$ follows from projective invariance of $\omega_{CFL}$.
\end{proof}

For the proof of Prop. \ref{prop:LerayCalc}, we require a few lemmas.
\begin{lem}
\label{lem:CauchyProjInv}
Let $\Omega$ be a pseudoconvex, piece-wise smooth domain, $\Hom{z} \in \sigma := \sigma(j_1, \ldots, j_n)$ with strong tangents $\Hom{w}^j$, $j=1, \ldots, n$. The expression 
\[ \frac{\det_0(\Hom{w}^j)}{\prod_j \langle \boldsymbol{\tau}, \Hom{w}^j \rangle} z_0^n dz \]
defines a projectively invariant $O_S(n,0) \otimes \sO_{\Omega}(-n,0)$-valued $n$ form.
\end{lem}

\begin{proof}
Since the strong tangents are mapped to each other by projective transformations which preserve $\langle, \rangle$, the expression $\frac{1}{\prod_j \langle \boldsymbol{\tau}, \Hom{w}^j \rangle}$ is invariant. The $n$-form $dz$ transforms like a section of $\sO_{S}(n+1,0)$, so it suffices to check that $\det_0(\Hom{w}^j)$ transforms like a section of $O_{S}(-1,0)$.
Let $T$ be a projective transformation with associated matrix $M$. It suffices to check invariance for $T$ in a generating set. \\
	Case 1: $T$ is a linear transformation in $z$, so $M$ is of the form 
	$$ \left(\begin{array}{@{}c|c@{}}
		\det(N)^{\frac{-1}{n+1}}
		& 0 \\
		\hline
	0 & \det(N)^{\frac{-1}{n+1}} N
	\end{array}\right), \ \ N \in GL_n(\C)$$
	Recalling that the induced transformation on the $\Hom{w}^j$ is $N^{-t}$,
	\begin{align*}
		T^*\left( \frac{\det_0(\Hom{w}^j)}{\prod_j \langle \boldsymbol{\tau}, \Hom{w}^j \rangle } dz \right) & = \frac{\det_0(T^{-t} \circ \Hom{w}^j)}{\prod_j \langle T^{-t}\Hom{w}^j, T\boldsymbol{\tau} \rangle } \det(N) dz \\
		& = \det(N)^{\frac{n}{n+1}}\frac{\det_0(\Hom{w}^j)}{\prod_j \langle \Hom{w}^j, \boldsymbol{\tau} \rangle } dz.
	\end{align*}
Case 2: $T$ is an affine shift,
\[ T(\Hom{z}) = [z_0:\ldots: z_{j-1}:z_j + \lambda z_0: z_{j+1} : \ldots z_n]. \]
Case 1 allows us to assume that $j = 1$. The dual transformation is given by 
\[ T^{-t}(\Hom{w}) = [w_0 - \lambda w_j: w_1 : \ldots : w_n] \]
so $T^*\det_0(\Hom{w}^j) = \det_0(T^{-t} \circ \Hom{w}^j) = \det_0(\Hom{w}^j)$. \\
Case 3: $T$ is an inversion, 
\[ T(\Hom{z}) = [z_0 + \lambda z_j: z_1 : \ldots : z_n]. \]
\indent Again, assume that $j = 1$. The dual transformation is given by 
\[ T^{-t}(\Hom{w}) = [w_0: w_1 - \lambda w_0: w_2 : \ldots : w_n] \]
Pulling back,
	\[ T^*({\det}_0(\Hom{w}^j))  = {\det}_0({\bf w}_j) - \lambda {\det}_1(\Hom{w}^j) \]
Let $W$ denote the matrix whose rows are the last $n$-coordinates of $\Hom{w}^j$ in affine coordinates, $\frac{w_k^j}{w^j_0}, j = 1, \ldots, n$. We have the matrix equation
\[ W \cdot z = -\boldsymbol{1} \]
An application of Cramer's rule gives $z_1 \det_0(W) = - z_0 \det_1(W)$, and by clearing denominators we get $z_0 \det_0(\Hom{w}^j) = -z_1 \det_1(\Hom{w}^j)$. It follows that 
\[ T^*({\det}_0(\Hom{w}^j)) = (1+\lambda \frac{z_1}{z_0}){\det}_0(\Hom{w}^j) \]
completing the proof.
\end{proof}

\begin{lem}{Simplex Calculation} \\
	\label{lem:simplexcalc}
	Let $e_j$ denote the $j^{th}$ standard basis vector, $\tau \in \C^n$. Suppose $w_j = w_j(t), t \in \Delta^n$ are functions satisfying $\sum w_i = 1$ and that \\
	$w_j(e_k) = \begin{cases}
		1 & j = k \\
		0 & j \neq k
	\end{cases}$. Then 
	\[  I(w_1, \ldots, w_n) := \int_{\Delta^n} \frac{1}{(1-\langle \tau, w \rangle)^n} dw_{[n]} = \frac{(-1)^{n}}{(n-1)!}\prod_{j=1}^{n} \frac{1}{1-\tau_j}. \]
\end{lem}

\begin{proof}
	The proof is induction on $n$. For $n = 2$, this reduces to calculating
	\[ \int_{t_1=0}^{t_1=1} \frac{1}{(1 - (w_1\tau_1 + w_2\tau_2))^2} dw_1. \]
	Keeping in mind that $w_2 = 1 - w_1$, we get
	\begin{align*}
		\left[\frac{-1}{(1-(w_1\tau_1 + w_2\tau_2))} \frac{-1}{\tau_1 - \tau_2} \right]_{t_1=0}^{t_1=1} = \frac{1}{(1-\tau_1)(1-\tau_2)}.
	\end{align*}
	For the induction step, we will use Stoke's theorem. Set $\partial \Delta_i^n := \Delta^n \cap \{t_i = 0\}$ oriented as the boundary of $\Delta^n$, and let $\varphi$ denote the integrand in the lemma. Thinking of $w_n$ as a function of the other $w_j$, we have $\frac{\partial w_n}{ \partial w_{n-1}} = - 1$. Define the primitive for $\varphi$ 
	\begin{align*}
		\eta := \frac{(-1)^{n-1}}{(n-1)(\tau_{n-1} - \tau_n )} \frac{1}{(1 - \langle \tau, w \rangle)^{n-1}} dw_1 \wedge \ldots \wedge dw_{n-2}.
	\end{align*}
	Note that $\eta$ is only non-zero along $\partial \Delta_i^n$ for $i = n-1, n$. Keeping in mind the orientation, we have
	\begin{align*}
		\int_{\partial \Delta_{n-1}^n} \eta & =  \frac{(-1)^{n-2}}{(n-1)} \frac{1}{\tau_{n-1} -\tau_n} (-1)^{n-2} I(w_1, \ldots, w_{n-2}, w_n) \\
		\int_{\partial \Delta_{n}^n} \eta & =  \frac{(-1)^{n-2}}{(n-1)} \frac{1}{\tau_{n-1} -\tau_n} (-1)^{n-1} I(w_1, \ldots, w_{n-2}, w_{n-1}).
	\end{align*}
	By Stoke's
	\begin{align*}
		\int_{\Delta^n} \varphi = \frac{1}{(n-1)(\tau_{n-1} - \tau_n)}\left( I(w_1, \ldots, w_{n-2}, w_n) - I(w_1, \ldots, w_{n-2}, w_{n-1}) \right).
	\end{align*}
	From induction, along with the identity
	\begin{align*}
		\prod_{j \neq n} \frac{1}{1- \tau_j} - \prod_{j \neq n-1} \frac{1}{1- \tau_j} = \frac{\tau_{n-1} - \tau_n}{\prod_{j=1}^n 1 - \tau_j}
	\end{align*}
	we are done.
\end{proof}

\begin{proof}{Proof of Prop. \ref{prop:LerayCalc}} \\
First we establish the following equality on $U := \{ z_0 \neq 0, z_n \neq 0, w_0 \neq 0\} \subset \Ps^n \times (\Ps^n)^*$
\begin{align}
\label{eq:CFLcalc}
\omega_{CFL} = z_0^n w_0^n \frac{(-1)^{\frac{n^2-n-2}{2}}(n-1)!}{(2\pi i)^n} \frac{1}{\hat{z}_n} dw_{[n]} \wedge dz
\end{align}
\par To see this, first make the substitution $\frac{1}{\hat{z}_n} = \hat{w}_n + \sum_{j=1}^{n-1} \frac{\hat{w}_j\hat{z}_j}{\hat{z}_n}$ into the right hand side. Then observe by differentiating the relation $z \cdot w = 1$, that $\hat{z}_jd\hat{w}_j = - \hat{z}_nd\hat{w}_n +$ terms that involve $d\hat{w}_j, j\neq n$ and $d\hat{z}_k$. These extra terms cancel upon substitution. The calculations described are below.
\begin{align*}
	\frac{1}{\hat{z}_n} & dw_{[n]} \wedge dz = (\hat{w}_n + \sum_{j=1}^{n-1} \frac{\hat{w}_j\hat{z}_j}{\hat{z}_n}) dw_{[n]} \wedge dz \\
	& = \hat{w}_n dw_{[n]} \wedge dz + \sum_{j=1}^{n-1}\hat{w}_j(-1)^{n-j} dw_{[j]} \wedge dz \\
	& = (-1)^n \sum_{j=1}^{n} (-1)^j \hat{w}_j dw_{[j]} \wedge dz
\end{align*}
\par To see that eq. (\ref{eq:CFLcalc}) agrees with the Universal CFL form, expand \\$(-1)^n \omega_{CFL}(\Hom{w}, \Hom{z})$ in the chart $U_{0,0} := \{ z_0 \neq 0, w_0 \neq 0\}$.
\par With this expression, it's clear that $\omega_{CFL}$ vanishes on edges of dimension $< n$ since we have $n$ differentials in $z$, taking values in an edge of dimension $<n$.
\par For the second calculation, use the projective invariance of both quantities to check the equality in a preferred set of coordinates. We will choose $z = (1, \ldots, 1)$, $\Hom{w}^j = [-1: 0 : \ldots : 0 : 1 : 0 : \ldots 0 ]$, where the $1$ is $j^{th}$ coordinate. Applying the previous calculation of $\omega_{CFL}$ along with Prop. \ref{prop:weaktangents} and Lemma \ref{lem:simplexcalc}, we have
\begin{align*}
\Leray & = \int_{W_{\Hom{z}}(\Omega)} \frac{1}{\pair{w}^n} \omega_{CFL} \\
& = z_0^n \frac{(-1)^{\frac{n^2-n-2}{2}}(n-1)!}{(2\pi i)^n} \left( \int_{\Delta^n} \tau_0^{-n} \frac{1}{(1 - \langle w, \tau \rangle)^n} dw_{[n]}\right) dz \\
& = z_0^n \frac{(-1)^{\frac{n^2+n-2}{2}}}{(2\pi i)^n} \tau_0^{-n} \frac{1}{\prod_j (1 - \tau_j)} dz \\
\end{align*}
The proof is complete up observing that, up to enumerating the strong tangents, this is exactly
\[ \frac{1}{(2 \pi i)^n}\frac{\det_0(\Hom{w}^j)}{\prod_j \langle \boldsymbol{\tau}, \Hom{w}^j \rangle} z_0^n dz \]
\end{proof}
Of course when $\sigma$ is dimension $2n-1$, $\Leray$ coincides with the usual Leray kernel. Now suppose $\gamma_j$ are a collection of analytic convex curves in $\C$ containing the origin, and $\Omega = \mathring{\gamma_1} \times \ldots \times \mathring{\gamma_n}$. It follows that the smooth portion of $S$ is \emph{Levi-flat}, meaning the defining functions $\rho_j = f_j + \overline{f_j}$ for $f_j$ holomorphic. It follows from Prop. \ref{def:weaktangents} that $\omega_{CFL}$ will vanish on edges of dimension $2n-k > n$, since the differentials $d\hat{w}_j$ are functions of holomorphic differentials in $z$ and at most $k-1$ other parameters. Thus we are left with the portion of $\Leray$ with $\Hom{z} \in \gamma_1 \times \ldots \times \gamma_n$. The strong tangents are given by $\Hom{w}^j = [-1: 0 : \ldots : 0 : \frac{1}{\hat{z}_j} : 0 \ldots 0]$, and we see that Thm. \ref{thm:LerayKernel} recovers the multi-variable Cauchy formula.

\section{An Invariant Hardy Space}
\par To construct a Hardy space, we would like a measure on $S$ such that a holomorphic function can be constructed by its $L^2$-boundary values with respect to this measure. From Thm \ref{thm:LerayKernel}, such a measure must give mass to edges of dimension $n$ to $2n-1$. Following \cite{barrett_2015}, we would also like the norm on this Hardy space to be projectively invariant. In this section, we produce such a measure when $n=2$.
\par When $n=2$, we need only consider edges of dimension two and three. For three-dimensional edges, there are several potential choices of measure, say Fefferman's measure or Barrett's preferred measure. Both measures are non-negative when $\Omega$ is locally strictly $\C$-convex, positive when $\Omega$ is strongly $\C$-convex, and satisfy the correct transformation law. We want to construct an analogous measure on edges of dimension $2$. To this end, we will say that a piece-wise smooth pseudoconvex domain is \emph{locally strictly $\C$-convex} if for each $\Hom{z} \in S$ there is a a neighborhood $U \ni \Hom{z}$ such that for all $\Hom{w} \in \weak$, $\ell_{\Hom{w}} \cap U \cap \overline{\Omega} = \{\Hom{z}\}$.
\par A first guess of a projectively invariant measure would be $|z_0^3dz|$. This is an $O_{\Ps^2}(3/2,3/2)$-valued 2-form, and so defines a norm on elements of $O_S(-3/2,0)$. However, elements of the Hardy space are going to be sections of $O_{S}(-2,0)$, and so we require a projectively invariant $O_S(2,2)$-valued 2-form. Fixing this requires finding a projectively invariant section $\eta \in O_{S}(j,j), \ j \neq 0$ which is non-negative when $\Omega$ is locally strictly $\C$-convex. The invariant measure will be given by $\mu := \eta^{\frac{1}{2j}} |z_0^3dz|$.
\par 
\begin{prop}
\label{prop:InvariantSection}
Let $\Omega \subset \Ps^2$ be a piece-wise smooth, pseudoconvex domain with all boundary hypersurfaces strongly $\C$-convex, and let $\sigma$ be a 2-dimensional edge. Then there is a projectively invariant section $\eta$ of \\ 
$O_{\sigma}(3/2,3/2)$ which is non-negative at points where $\Omega$ is locally strictly $\C$-convex. Furthermore, if $\eta > 0$ at a point $\Hom{z} \in \sigma$ then $\Omega$ is locally strictly $\C$-convex at $\Hom{z}$.
\end{prop}

\begin{proof}
First we describe how to construct $\eta$. Let $H_1, H_2$ be boundary hypersurfaces with $\sigma = H_1 \cap H_2$. Given $\Hom{z} \in \sigma$, make a change of coordinates so that $z = 0$ and the tangent cone of $\Omega$ at $0$ coincides with $\{Im(\hat{z}_1) < 0, Im(\hat{z}_2) < 0\}$. Writing $\hat{z}_j = x_j + iy_j$, on the hypersurface $H_1$ we have
\begin{align}
\label{eq:H1}
y_1 = \alpha_1|\hat{z}_2|^2 + Re(\beta_1 \hat{z}_2^2) + b_1x_1x_2 + \gamma_1x_1y_2 +a_1^2x_1^2 + h.o.t.
\end{align}
and the same on $H_2$
\begin{align}
\label{eq:H2}
y_2 = \alpha_2|\hat{z}_1|^2 + Re(\beta_2 \hat{z}_1^2) + b_2x_1x_2 + \gamma_2x_2y_1 +a_2^2x_2^2 + h.o.t.
\end{align}
Strong $\C$-convexity of each hypersurface implies $\alpha_j + |\beta_j| < 0$. Along $\sigma$, we can expand the $y_j$ in terms of the $x_j$
\begin{align*}
    y_1 & = a_1x_1^2 + b_1x_1x_2 + c_1x_2^2 + h.o.t.\\
    y_2 & = a_2x_2^2 + b_2x_1x_2 + c_2x_1^2 + h.o.t.
\end{align*}
It follows that $c_j = \alpha_j + Re(\beta_j) < 0$. This choice of coordinates is specified up to the subgroup of transformations generated by 
\begin{align*}
    I_\lambda := \begin{pmatrix}
    1 & 0 & 0 \\
    \lambda & 1 & 0 \\
    0 & 0 & 1 \end{pmatrix} \ \lambda \in \C, \ S_r := \begin{pmatrix}
    1 & 0 & 0 \\ 0 & r & 0 \\ 0 & 0 & 1 \end{pmatrix} \ r \in \mathbb{R}_{>0}, \ W := \begin{pmatrix} 1 & 0 & 0 \\ 0 & 0 & 1 \\ 0 & 1 & 0 \end{pmatrix}
\end{align*}
\par Using Mathematica, we compute how the normal forms change under application of these matrices. The calculations are included in the Appendix, and we record them here. Identifying the normal form above with the matrix
\begin{align*}
    \begin{pmatrix}
    a_1 & b_1 & c_1 \\
    a_2 & b_2 & c_2
    \end{pmatrix},
\end{align*}
we compute
\begin{align*}
    I_\lambda^* \begin{pmatrix}
    a_1 & b_1 & c_1 \\
    a_2 & b_2 & c_2
    \end{pmatrix} & = \begin{pmatrix}
    a_1 + Im(\lambda) & b_1 & c_1 \\
    a_2 & b_2 + Im(\lambda) & c_2
    \end{pmatrix} \\
    S_r^* \begin{pmatrix}
    a_1 & b_1 & c_1 \\
    a_2 & b_2 & c_2
    \end{pmatrix} & = \begin{pmatrix}
    a_1r & b_1 & \frac{c_1}{r} \\
    a_2 & b_2r & c_2r^2
    \end{pmatrix} \\
    W^*\begin{pmatrix}
    a_1 & b_1 & c_1 \\
    a_2 & b_2 & c_2
    \end{pmatrix} & = \begin{pmatrix}
    a_2 & b_2 & c_2 \\
    a_1 & b_1 & c_1
    \end{pmatrix}
\end{align*}
From these transformation laws, it follows that there is a unique (up to $W$) set of projective coordinates such that $a_1 = a_2 = 0$ and $c_2 = c_1 = -1$. Now suppose every weak tangent, which are given by $\Hom{w} = [0:t:1-t], t \in [0,1]$, avoids $\Omega$ in a uniform neighborhood of $\Hom{z}$. In particular, each weak tangent, given by the equation $tz_1 + (1-t)z_2 = 0$ must avoid $\sigma$ in a neighborhood of 0. Fixing $t$, and setting $s = 1-t$, this implies that whenever $x_2 = \frac{-t}{s}x_2$, we must have 
\begin{align*}
    0 \neq ty_1 + sy_2
\end{align*}
We will show that the above expression, expanded in terms of $x_1$, has lowest order term of degree 2, and that the coefficient is negative as $t \rightarrow 0$. Thus to be non-zero, the coefficient must be non-positive for all $t$. We now go through the calculations.
\par Expanding the right hand side in terms of $x_1$,
\begin{align*}
    ty_1 + sy_2 = (-t^2s^{-1}b_1-t^3s^{-2}-tb_2-s)x_1^2 + o(|x_1|^3)
\end{align*}
Note that the coefficients in the higher order terms are bounded as $t \rightarrow 0$.
\par Multiplying by $\frac{s^2}{x_1^2}$, we see that this is equivalent to
\[ 0 \neq -t^3 - s^3 -t^2sb_1 - ts^2b_2, \ t \in [0,1] \]
Rearranging slightly, this is equivalent to 
\[ -tb_1-sb_2 < \frac{t^3 + s^3}{st} \]
The function $f(t) = \frac{t^3 + s^3}{st}$ is defined on $(0,1)$ and easily verified to be convex, while the left hand side is a linear function who's endpoints have height $-b_2, -b_1$. To make the situation symmetric about the origin, introduce the change of variables $\hat{t} = 2t-1$, $\hat{s} = -\hat{t} = 2s-1$, and we obtain
\begin{align}
\label{eq:Legendre_inequality}
-\frac{\hat{t}+1}{2}b_1 - \frac{\hat{s}+1}{2}b_2 < \frac{(\frac{\hat{t}+1}{2})^3 + (\frac{\hat{s}+1}{2})^3}{(\frac{\hat{t}+1}{2})( \frac{\hat{s}+1}{2})}
\end{align}

\begin{center}
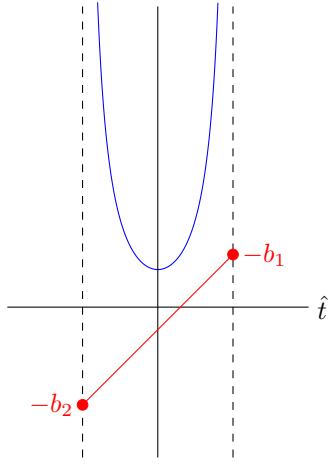
\begin{figure}
\begin{tikzpicture}
\filldraw (2, 0) circle (0pt) node[anchor = west]{$\hat{t}$};
\draw[smooth,samples=100,domain=-.8:.8,blue] plot(\x,{((\x+1)^3 + (1-\x)^3)/(4*(1+\x)*(1-\x))});
\draw[dashed] (-1, -2) -- (-1, 4);
\draw[dashed] (1, -2) -- (1, 4);
\draw (-2,0) -- (2,0);
\draw (0,-2) -- (0,4);
\filldraw[red] (-1,-1.3) circle (2pt) node[anchor=east]{$-b_2$};
\filldraw[red] (1,.7) circle (2pt) node[anchor=west]{$-b_1$};
\draw[red] (-1,-1.3) -- (1, .7);

\end{tikzpicture}
\caption{Visualization of inequality describing weak tangents locally avoiding $\sigma$. The graph of $f$ is in blue, the left hand side of the inequality is represented by the red line.}
\label{fig:Inequality}
\end{figure}
\end{center}

We see that the inequality holds for all $\hat{t} \in [-1, 1]$ exactly when the Legendre transform of $f$ in the $\hat{t}$ variable, denoted $L\{f\}$, satisfies
\[ L\{f\}(\frac{-b_1+b_2}{2}) \leq \frac{b_1 + b_2}{2} \]
See Fig. \ref{fig:Inequality} to make this apparent. It follows from the symmetry about the origin that this quantity is invariant with respect to exchanging $b_1$ and $b_2$. Therefore, the quantity 
\[ \kappa = \frac{b_1+b_2}{2} - L\{f\}(\frac{-b_1+b_2}{2}) \]
is projectively invariant. It follows from the comment that the higher order terms are bounded as $t \rightarrow 0$ that this quantity is non-negative whenever $\Omega$ is locally strictly $\C$-convex. A symmetric argument applies as $t \rightarrow 1$.
\par From the calculations above, we see that $\eta := c_1c_2\kappa |z_0|^3$ is the desired projectively invariant section of $O_S(3/2,3/2)$. It remains to show that when $\eta >0$, $\Omega$ is locally strictly $\C$-convex. Suppose this is the case, and let $z$ be a point near $0$ contained in the weak tangent $(t, s) \cdot z = 0$. Suppose for contradiction that $z \in \overline{\Omega}$. From the proof of Prop. \ref{prop:weaktangents}, there are constants $C_j$ so that $|y_j| \leq C_j||(x_1, x_2)||^2$. Plugging in these bounds to eq. (\ref{eq:H1}) and (\ref{eq:H2}), and assuming we are in coordinates where $a_1 = a_2 = 0, c_1 = c_2 = -1$, we have

\begin{align*}
    y_1 \leq b_1x_1x_2-x_2^2 + h.o.t. \\
    y_2 \leq b_2x_1x_2-x_1^2 + h.o.t.
\end{align*}
Since $\eta > 0$, we have $ty_1 + sy_2 < 0$, which is the desired contradiction.
\end{proof}

\begin{rem}
\par The choice to Legendre transform is somewhat arbitrary; we are in a situtation where the choice of coordinates is unique up to permutation, and so there are many different ways to characterize inequality (\ref{eq:Legendre_inequality}) being satisfied for all $\hat{t}$. The Legendre transformation was chosen simply because of the similarity between the Legendre transform from real convexity theory and the dual map in complex variables (see Remark 18 of \cite{barrett_2015}).
\end{rem}

\begin{ex}{Perturbation of the Bidisk}
\par For the bidisk $\D^2$, there is one 2-dimensional edge, namely the distinguished boundary $b\D^2 = S^1 \times S^1$. Taking coordinates centered at $(1,1)$, the expansion along $b\D^2$ up to order 2 is
\begin{align*}
    y_1 & = -\frac{1}{2}x_1^2 \\
    y_2 & = -\frac{1}{2}x_2^2 \\
\end{align*}
\par $\D^2$ fails to be strongly $\C$-convex along the smooth portion of the boundary, but one can take a strongly $\C$-convex approximation via
\[ \Omega_{\epsilon} := \{|z_1|^2 + \epsilon|z_2|^2 < 1, |z_2|^2 + \epsilon|z_1|^2 < 1\}, \ 0 < \epsilon << 1 \]
It follows that the normal from along the 2-dimensional edge of $\Omega$ will have $a_j < 0$, and $b_j$ small, and we can conclude that $\eta > 0$.
\end{ex}
\par An interesting and related question is the following; Let $S \subset \Ps^n$ be a totally real $n$-manifold. What are the local projective invariants of $S$? This question has been studied for hypersurfaces for holomorphic and projective changes of coordinates, see \cite{Chern_Moser_1974}, \cite{Bolt_2008}, and \cite{Hammond_Robles_2013}.
\par When $n=2$, one can make a change of coordinates so that $z = 0$, $T_0(S) = \{y_1 = y_2 = 0\}$, which specifies the coordinates up to $I_\lambda, \lambda \in \C, S_r, r \in \mathbb{R}^*, W$ and the matrix
\begin{align*}
    H_r := \begin{pmatrix}
    1 & 0 & 0 \\
    0 & 1 & r \\
    0 & 0 & 1
    \end{pmatrix}, \ r \in \mathbb{R}
\end{align*}
We compute the pullback of the normal form in the Appendix, and record it here.
\begin{align*}
   H_r^* \begin{pmatrix}
    a_1 & b_1 & c_1 \\
    a_2 & b_2 & c_2
    \end{pmatrix} =
    \left(
    \begin{matrix}
    a_1 +rb_1+r^2c_1 & b_1+2rc_1 \\
    a_2-rc_1 & b_2+2ra_2-rb_1+2r^2c_1 \\
    \end{matrix} \right. \\
    \left. \begin{matrix}
    c_1 \\
    c_2+r(b_2-a_1)+r^2(a_2+b_1)+r^3c_1
    \end{matrix} \right)
\end{align*}

\par Keeping Prop. \ref{prop:InvariantSection} in mind, we say that a piece-wise smooth, pseudoconvex domain $\Omega \subset \Ps^2$ is \emph{strongly $\C$-convex} if $\Omega$ is strictly $\C$ convex, each boundary hypersurface is strongly $\C$-convex, and $\eta > 0$ on each 2-dimensional edge.
\par Let $\tilde{\mu}$ be an positive, projectively invariant $O_S(-2,-2)$-valued 3-form defined on the 3-dimension edges of $S$, such as the Fefferman form or Barrett's preferred measure. Let $\mu' := \eta^{\frac{1}{3}} |z_0^3dz|$, and set $\mu = \tilde{\mu} + \mu'$. Define the norm of a section $f$ of $O_{S}(-n,0)$ to be 
\[ ||f||^2 = \int_S |f|^2 d\mu. \]
By construction, this norm is projectively invariant. Let $L^2(\mu)$ be the corresponding completion, and define the Hardy space
\[ H^2(\Omega, \mu) := \overline{\{f|_S | f \in \sO_{\overline{\Omega}}(-n,0)\}} \]
The following theorem, which follows directly from strict positivity of $\eta$ and Thm \ref{thm:LerayKernel}, justifies calling this the Hardy space.
\begin{thm}
Let $\Omega \subset \Ps^2$ be piece-wise smooth, strongly $\C$-convex domain. The restriction $\sO_{\overline{\Omega}}(-2,0) \rightarrow H^2(\Omega, \mu)$ is injective, with inverse provided by integrating against the kernel $\Leray$.
\end{thm}

\thanks{The author would like to thank David Barrett for his mentorship and guidance while preparing this work.}


%

\bibliographystyle{amsalpha}
\bibliography{main}

\setcounter{secnumdepth}{0}
\appendix
\section{Appendix: Projective Normal Form Calculations}
\label{appendix}
\begin{doublespace}
	\noindent\(\pmb{\text{PullBack}[\text{expr1$\_$}, \text{expr2$\_$}, \text{vars$\_$},\text{M$\_$}] \text{:=} \text{Module}[\{\text{z1}, \text{z2},
		\text{w1}, \text{w2},\text{x1},\text{y1},\text{x2},\text{y2},t,} \\
	\indent \pmb{\text{newexpr1},\text{newexpr2}\},}\\
	\indent \pmb{\text{z1} = \text{x1} + I \text{y1};}\\
	\indent \pmb{\text{z2} = \text{x2} + I \text{y2};}\\
	\indent \pmb{\text{w1} = (M[[1,2]]+M[[2,2]]*\text{z1}+M[[3,2]]*\text{z2})/(M[[1,1]]+M[[2,1]]*\text{z1}}\\
	\indent \indent \pmb{+M[[3,1]]*\text{z2});}\\
	\indent \pmb{\text{w2} = (M[[1,3]]+M[[2,3]]*\text{z1}+M[[3,3]]*\text{z2})/(M[[1,1]]+M[[2,1]]*\text{z1}}\\ \indent \indent \pmb{+M[[3,1]]*\text{z2});}\\
	\indent \pmb{\text{(*}\text{Substituting} \text{ in} \text{ Real}, \text{ Imaginary} \text{ parts} \text{ of} \text{ transformed} \text{ variables} \text{ into} \text{ the}}\\
	\indent \pmb{\text{ normal} \text{ form}\text{*)}}\\
	\indent \pmb{\text{newexpr1} = \text{expr1}\text{/.}\text{Thread}[\text{vars}\text{-$>$}\{\text{ComplexExpand}[\text{Re}[\text{w1}]],
	}\\
    \indent \indent \pmb{\text{ComplexExpand}[\text{Re}[\text{w2}]], \text{ComplexExpand}[\text{Im}[\text{w1}]],}\\ \indent \indent \pmb{\text{ComplexExpand}[\text{Im}[\text{w2}]]\}];}\\
	\indent \pmb{\text{newexpr2} =\text{expr2}\text{/.}\text{Thread}[\text{vars}\text{-$>$}\{\text{ComplexExpand}[\text{Re}[\text{w1}]],}\\ \indent \indent \pmb{\text{ComplexExpand}[\text{Re}[\text{w2}]],\text{ComplexExpand}[\text{Im}[\text{w1}]],}\\
	\indent \indent \pmb{\text{ComplexExpand}[\text{Im}[\text{w2}]]\}];}\\
	\pmb{}\\
	\indent \pmb{\text{(*}\text{Discarding} \text{ all} \text{ terms} \text{ of} \text{ degree} >2, \text{ recall} \text{ y1}, \text{ y2} \text{ have}
		\text{ degree } 2\text{*)}}\\
	\indent \pmb{\text{newexpr1} = \text{Normal}[\text{Series}[\text{newexpr1}\text{/.} \text{Thread}[\{\text{y1},\text{y2},\text{x1},\text{x2}\}\text{-$>$}\{t{}^{\wedge}2*\text{y1},}\\
	\indent \indent \pmb{t{}^{\wedge}2*\text{y2}, t*\text{x1},t*\text{x2}\}],\{t,0,2\}]]\text{/.} t\text{-$>$}1; }\\
	\indent \pmb{\text{newexpr2} =\text{Normal}[\text{Series}[\text{newexpr2}\text{/.} \text{Thread}[\{\text{y1},\text{y2},\text{x1},\text{x2}\}\text{-$>$}\{t{}^{\wedge}2*\text{y1},}\\
	\indent \indent \pmb{t{}^{\wedge}2*\text{y2}, t*\text{x1},t*\text{x2}\}],\{t,0,2\}]]\text{/.} t\text{-$>$}1;}\\
	\pmb{}\\
	\indent\pmb{\text{(*Substituting in original variables*)}}\\
	\indent \pmb{\{\text{Collect}[\text{newexpr1}\text{/.}\text{Thread}[\{\text{x1},\text{x2},\text{y1},\text{y2}\}\text{-$>$}\text{vars}],\text{vars}], \text{Collect}[\text{newexpr2}}\\
	\indent \indent \pmb{\text{/.}\text{Thread}[\{\text{x1},\text{x2},\text{y1},\text{y2}\}\text{-$>$}\text{vars}],\text{vars}]\}}\\
	\pmb{]}\)
\end{doublespace}

\pmb{\text{(* Pulling back by $I_\lambda, \lambda \in \mathbb{R}$}*)}

\begin{doublespace}
	\noindent\(\pmb{\text{PullBack}[\text{a1}*\text{s1}{}^{\wedge}2+\text{b1}*\text{s1}*\text{s2}+\text{c1}*\text{s2}{}^{\wedge}2-\text{t1},\text{a2}*\text{s2}{}^{\wedge}2+\text{b2}*\text{s1}*\text{s2}+}\\
	\indent \pmb{\text{c2}*\text{s1}{}^{\wedge}2-\text{t2},\{\text{s1},\text{s2},\text{t1},\text{t2}\},\{\{1,0,0\},\{\lambda ,1,0\},\{0,0,1\}\}]}\)
\end{doublespace}

\begin{doublespace}
	\noindent\(\left\{\text{a1} \text{s1}^2+\text{b1} \text{s1} \text{s2}+\text{c1} \text{s2}^2-\text{t1},\text{c2} \text{s1}^2+\text{b2} \text{s1} \text{s2}+\text{a2}
	\text{s2}^2-\text{t2}\right\}\)
\end{doublespace}

\pmb{\text{(* Pulling back by $I_\lambda, \lambda \in i\mathbb{R}$}*)}

\begin{doublespace}
	\noindent\(\pmb{\text{PullBack}[\text{a1}*\text{s1}{}^{\wedge}2+\text{b1}*\text{s1}*\text{s2}+\text{c1}*\text{s2}{}^{\wedge}2-\text{t1},\text{a2}*\text{s2}{}^{\wedge}2+\text{b2}*\text{s1}*\text{s2}+}\\
	\indent \pmb{\text{c2}*\text{s1}{}^{\wedge}2-\text{t2},\{\text{s1},\text{s2},\text{t1},\text{t2}\},\{\{1,0,0\},\{I*\lambda ,1,0\},\{0,0,1\}\}]}\)
\end{doublespace}

\begin{doublespace}
	\noindent\(\left\{\text{b1} \text{s1} \text{s2}+\text{c1} \text{s2}^2-\text{t1}+\text{s1}^2 (\text{a1}+\lambda ),\text{c2} \text{s1}^2+\text{a2}
	\text{s2}^2-\text{t2}+\text{s1} \text{s2} (\text{b2}+\lambda )\right\}\)
\end{doublespace}

\pmb{\text{(* Pulling back by $S_r$}*)}
\(\)

\begin{doublespace}
	\noindent\(\pmb{\text{PullBack}[\text{a1}*\text{s1}{}^{\wedge}2+\text{b1}*\text{s1}*\text{s2}+\text{c1}*\text{s2}{}^{\wedge}2-\text{t1},\text{a2}*\text{s2}{}^{\wedge}2+\text{b2}*\text{s1}*\text{s2}+}\\
	\indent \pmb{\text{c2}*\text{s1}{}^{\wedge}2-\text{t2},\{\text{s1},\text{s2},\text{t1},\text{t2}\},\{\{1,0,0\},\{0,r,0\},\{0,0,1\}\}]}\)
\end{doublespace}

\begin{doublespace}
	\noindent\(\left\{\text{a1} r^2 \text{s1}^2+\text{b1} r \text{s1} \text{s2}+\text{c1} \text{s2}^2-r \text{t1},\text{c2} r^2 \text{s1}^2+\text{b2}
	r \text{s1} \text{s2}+\text{a2} \text{s2}^2-\text{t2}\right\}\)
\end{doublespace}

\pmb{\text{(* Pulling back by $H_r$}*)}

\begin{doublespace}
	\noindent\(\pmb{\text{res} =\text{PullBack}[\text{a1}*\text{s1}{}^{\wedge}2+\text{b1}*\text{s1}*\text{s2}+\text{c1}*\text{s2}{}^{\wedge}2-\text{t1},\text{a2}*\text{s2}{}^{\wedge}2+}\\
	\indent \pmb{\text{b2}*\text{s1}*\text{s2}+\text{c2}*\text{s1}{}^{\wedge}2-\text{t2},\{\text{s1},\text{s2},\text{t1},\text{t2}\},\{\{1,0,0\},\{0,1,r\},\{0,0,1\}\}]}\)
\end{doublespace}

\begin{doublespace}
	\noindent\( \left(\text{a1}+\text{b1} r+\text{c1} r^2\right) \text{s1}^2+(\text{b1}+2 \text{c1} r) \text{s1} \text{s2}+\text{c1} \text{s2}^2-\text{t1},\left(\text{c2}+\text{b2}
	r+\text{a2} r^2\right) \text{s1}^2+\)
\end{doublespace}

\begin{doublespace}
\indent \((\text{b2}+2 \text{a2} r) \text{s1} \text{s2}+\text{a2} \text{s2}^2-r \text{t1}-\text{t2} \)
\end{doublespace}

\begin{doublespace}
	\noindent\(\pmb{\text{Collect}[\text{res}[[2]]-r*\text{res}[[1]], \{\text{s1},\text{s2},\text{t1},\text{t2}\}]}\)
\end{doublespace}

\begin{doublespace}
	\noindent\(\left(\text{c2}+\text{b2} r+\text{a2} r^2-r \left(\text{a1}+\text{b1} r+\text{c1} r^2\right)\right) \text{s1}^2+(\text{b2}+2 \text{a2}
	r-r (\text{b1}+2 \text{c1} r)) \text{s1} \text{s2}+(\text{a2}-\text{c1} r) \text{s2}^2-\text{t2}\)
\end{doublespace}

\begin{doublespace}
	\noindent\(\pmb{\text{}}\)
\end{doublespace}

\end{document}